\pdfoutput=1

\documentclass[a4paper, 11pt]{article}

\usepackage[english]{babel}
\usepackage{enumitem}
\usepackage[dvipsnames]{xcolor}
\usepackage{multirow}
\usepackage[utf8]{inputenc}
\usepackage{apacite}
\usepackage{natbib}
\usepackage{adjustbox}
\usepackage{amsmath, amsthm}
\usepackage{amsfonts}
\usepackage{graphicx}
\usepackage{subcaption}
\usepackage{pgfplots}
\pgfplotsset{compat=1.12}
\usepackage{eurosym}
\newtheorem{prop}{Proposition}

\usepackage{setspace}
\usepackage[margin=1 in]{geometry}

\usepackage{algorithm,algpseudocode}

\usepackage[short,nocomma]{optidef}

\usepackage{authblk}

\begin{document}

\title{Dynamic Discretization Discovery for the Multi-Depot Vehicle Scheduling Problem with Trip Shifting}
\date{\today}
    
    \author[1,*]{R.N. van Lieshout}
	\author[2]{D.T.J. van der Schaft}

	\affil[1]{\footnotesize Department of Operations, Planning, Accounting, and Control, School of Industrial Engineering, Eindhoven University of Technology, Eindhoven, The Netherlands}
	\affil[2]{\footnotesize Integral Capacity Management, Erasmus Medical Center, Rotterdam, The Netherlands}
	\affil[*]{\footnotesize Corresponding author. Email: r.n.v.lieshout@tue.nl}

		\maketitle

\begin{abstract}
\normalsize
The solution of the Multi-Depot Vehicle Scheduling Problem (MDVSP) can often be improved substantially by incorporating Trip Shifting (TS) as a model feature. By allowing departure times to deviate a few minutes from the original timetable, new combinations of trips may be carried out by the same vehicle, thus leading to more efficient scheduling. However, explicit modeling of each potential trip shift quickly causes the problem to get prohibitively large for current solvers, such that researchers and practitioners were obligated to resort to heuristic methods to solve large instances. In this paper, we develop a Dynamic Discretization Discovery algorithm that guarantees an optimal continuous-time solution to the MDVSP-TS without explicit consideration of all trip shifts. It does so by iteratively solving and refining the problem on a partially time-expanded network until the solution can be converted to a feasible vehicle schedule on the fully time-expanded network. Computational results demonstrate that this algorithm outperforms the explicit modeling approach by a wide margin and is able to solve the MDVSP-TS even when many departure time deviations are considered.
\vspace{15pt}

\end{abstract}

\pagenumbering{arabic}

\section{Introduction}
Public transportation companies aim to provide service to passengers in a certain region by allocating vehicles to take these passengers from one place to another. The field of vehicle scheduling deals with optimally combining a predefined set of timetables into a schedule, adhering to the availability of vehicles at the right moments in time at specified locations. When vehicles can be pulled out of multiple depots, the problem of minimizing operating costs whilst covering every trip is called the Multi-Depot Vehicle Scheduling Problem (MDVSP).

Although the provision of timetables beforehand heavily reduces the complexity of the problem, it limits the opportunities to improve the solution quality. Even slight alterations of the starting times of trips may allow connections to be made that were impossible at first, potentially enhancing the schedule and reducing operating costs. Allowing the timetables to deviate from its original values is referred to as trip shifting, transforming the problem into a Multi-Depot Vehicle Scheduling Problem with Trip Shifting (MDVSP-TS).

Ideally, we would like to consider all starting times within a certain margin from the timetable. Unfortunately, the more allowed deviations, the larger the set of possible alternatives becomes, up to a point where the mathematical model is prohibitively large for current solvers. In practice, heuristic methods are applied to large instances instead at the cost of guaranteed optimality.

A promising exact solution method that may alleviate the reliance on heuristic approaches for the MDVSP-TS is Dynamic Discretization Discovery (DDD; Boland et al., 2017). This recently developed technique operates on a partially time-expanded network, overcoming the need to explicitly model all possible trip shifts, yet yielding optimal continuous-time solutions. In an iterative fashion, the outcome of an IP-model based on a partial discretization is used to discover new time points and refine the network accordingly. \cite{boland2017continuous} apply the algorithm to the Service Network Design problem, but as the time-expanded network structure of this problem is prevalent in many transportation problems, the authors expect that many of the fundamental ideas underlying their approach are widely applicable.

Apart from the theoretical interest of developing a new solution approach for the MDVSP-TS, the computational study conducted by \cite{boland2017continuous} shows the potential value in practical sense as DDD outperforms other exact methods for solving the Service Network Design problem. Moreover, after its introduction, DDD has been investigated in the context of fundamental optimization problems such as the shortest path problem \citep{he2021dynamic} and the traveling salesman problem \citep{vu2020dynamic}, time after time proving to be highly effective. Therefore, the goal of this paper is to apply DDD to the MDVSP-TS and investigate its potential.


In this paper, we develop a DDD algorithm that is specifically designed for the MDVSP-TS. A partially time-expanded network is constructed that contains a path for every feasible sequence of trips by underestimating the length of arcs. Then, in iterative fashion, the problem is solved on the partial network and converted to a solution that adheres to the actual travel times of the arcs. If this solution is feasible, optimality has been reached. Otherwise, time points are added to the partial network that cut off the current solution, and the updated model is solved.

Within the DDD solution scheme, we propose three methods to construct the partially time-expanded network. The first method best resembles other DDD approaches, where all arcs are either the correct length, or too short. The second and third method take advantage of the structure of the MDVSP-TS to lengthen arcs without increasing the number of time points in the network, resulting in stronger relaxations. In addition, we propose three refinement strategies. The first refinement strategy transforms the solution on the partial network into duties, and then evaluates those duties one by one. The other refinement strategies exploit that since the partial network is relatively coarse, the decomposition of its solution into duties is typically non-unique. Hence, these strategies \textit{optimize} the duty decomposition, one targeted at keeping the network size as small as possible, and the other aimed at reducing the number of DDD iterations. The non-uniqueness of the duty decomposition can also be leveraged to alleviate the impact of trip shifting on the timetable in a post-processing step. 

To test the performance of DDD, we conduct a series of computational experiments on instances of 500, 750 and 1000 trips generated according to the data generation process introduced by \cite{carpaneto1989branch}. The results show that, provided that one uses the stronger network generation methods, DDD is able to find an optimal solution to the MDVSP-TS orders of magnitude faster than the explicit modeling technique that operates on a fully time-expanded network. This allows for a wider range of departure times to be considered for the trips, resulting in a significant reduction in objective value. Moreover, the results show that optimizing the duty decomposition as a post-processing step enables the elimination of unnecessary trip shifts, such that the impact on the timetable is limited.



The paper is organized as follows: Section \ref{s:literature} provides an overview of the existing research on the MDVSP-TS and DDD. Then, Section \ref{s:mdvsp} gives a description of the MDVSP-TS, and Section \ref{s:ddd} introduces the DDD algorithm that is developed for the MDVSP-TS. The computational results of the algorithm are presented in Section \ref{s:results}. Finally, conclusions on the performance of DDD are drawn in Section~\ref{s:conclusion}.
\section{Literature Review}
\label{s:literature}

\subsection{Multi-Depot Vehicle Scheduling and Timetabling}
\label{ss:lit_mdvsp}
The class of vehicle scheduling problems has been extensively researched for more than 50 years in Operations Research \citep{bunte2009overview}. In short, it revolves around finding a feasible assignment of timetabled trips to vehicles such that each trip is covered once and operating costs are minimized. The consideration of multiple depots where vehicles must start and finish their route through the network heavily increases the complexity of the problem. Whereas multiple polynomial time algorithms have been developed for the single depot case, the problem considering multiple depots is proven to be NP-hard \citep{bertossi1987some}. 

Modeling the multiple depot case is generally done in either of the following three ways: by means of a single-commodity, multi-commodity or set partitioning model. First of all, in the single-commodity representation of the problem, both trips and vehicles are represented as nodes. The goal is to construct a set of circuits containing one vehicle node and one or more trip nodes that minimizes the operating costs. \cite{carpaneto1989branch} were the first to develop a branch-and-bound algorithm that solves the MDVSP using such a single-commodity approach. \cite{mesquita1992multiple} improved upon the original representation of the problem by aggregating all vehicle nodes corresponding to the same depot to arrive at a more comprehensive network structure and a considerable reduction in the amount of variables and constraints. Secondly, the multi-commodity formulations are based on a layered network where each layer represents the network with regards to one of the depots. A distinction can be made between so-called connection-based networks and time-space networks. In the former, each possible connection between two compatible trips is explicitly modeled. Solution methods that employ the connection-based representation include a branch-and-bound algorithm by \cite{forbes1994exact} and two heuristic approaches by \cite{haghani2002heuristic} reducing the problem size by imposing fuel consumption constraints. In the time-space network representation, the number of arcs in the network connecting compatible trips can be heavily reduced by aggregation of deadheading opportunities (97\% to 99\% on real-world instances) without compromising the feasible solution space \citep{kliewer2006time}. Finally, set partitioning models aim to select a subset of feasible routes such that all trips are covered exactly once at minimal cost. The routes are chosen from a set of feasible routes that is either constructed a priori or gradually by means of column generation. The set partitioning approach was first introduced by \cite{ribeiro1994column} who use column generation to iteratively add reduced cost vehicle routes to the mathematical model.

The opportunity to alter departure times of scheduled trips in the MDVSP was first examined by \cite{mingozzi1995exact}, who employed a branch and bound algorithm to solve a set partitioning formulation of the MDVSP including time windows. \cite{desaulniers1998multi} used a multi-commodity network flow representation with a non-linear objective function to more realistically interpret waiting costs. Their column generation method was later adopted by \cite{hadjar2009dynamic} and improved through tactical reduction of the width of the time windows. \cite{kliewer2006timeB} were the first to consider trip shifting in the time-space network formulation of the MDVSP.

Extending the MDVSP with time windows or trip shifting for all trips quickly deems standard exact solution methods intractable for large-sized instances. Therefore, heuristic approaches have been developed to deal with alterations to the departure time. Examples include heuristics that construct a small set of trip shifts with high cost saving potential \citep{kliewer2006timeB} and a two-phase matheuristic that builds bus timetables in the first stage and constructs an optimal schedule given these timetables thereafter \citep{desfontaines2018multiple}. Despite the computational efficiency of such methods, the researchers of the latter underline the relevance of integrating both trip shifting and vehicle scheduling into an iterative algorithm.

\cite{petersen2013simultaneous} considered the quality of the timetable from a passenger's perspective while also scheduling the vehicles by partly evaluating solutions based on the waiting time between different lines in a large neighbourhood search algorithm. \cite{schmid2015integrated} employed a similar approach and decomposed the problem into a scheduling and a balancing component. They applied large neighbourhood search to the scheduling phase of the problem, whereas departure times are balanced through solving a linear program. The importance of timetable quality from the viewpoint of the passenger has been further investigated by \cite{laporte2017multi} and \cite{liu2017integrated} by taking passenger connections into account besides timetabling and vehicle scheduling. Both studies obtain a set of Pareto-efficient solutions through an $\epsilon$-constraint method and a deficit function-based sequential search method, respectively. \cite{fonseca2018matheuristic} developed a matheuristic that iteratively adds modified departure times to the integrated timetabling and vehicle scheduling problem by either one of four selection strategies, allowing a wider set of timetable adjustments to be considered.

A related stream of literature considers the integration of timetabling and vehicle scheduling in the context of periodic timetables and vehicle schedules (i.e., circulations). \cite{kroon2014flexible} developed an approach that counts the number of required vehicles within the standard periodic timetabling framework model. However, this required assuming that at most one vehicle is present at each station at a time. \cite{van2021integrated} relaxed this assumption and obtained the set of Pareto-efficient solutions with respect to the number of required vehicles and the average perceived travel time of passengers. Also in this context, simultaneous optimization of the timetable and vehicle schedule proved highly beneficial, as it allowed considerable reductions in the number of vehicles, at the expense of marginal increases in passenger travel time. 

\subsection{Dynamic Discretization Discovery}
\label{ss:lit_ddd}
The DDD algorithm by \cite{boland2017continuous} has originally been developed for the Continuous-Time Service Network Design problem, in which minimum-cost paths are constructed in both space and time for the shipments of goods and the resources necessary to execute these shipments. Because a full discretization of the problem (i.e. taking all possible dispatch times into consideration) quickly causes the network to grow too large to solve the problem in a reasonable amount of time, coarser discretizations are commonly used to ensure the computational tractability. \cite{boland2017continuous} were the first to prove that an optimal continuous-time solution can be found without full discretization, but through iterative refinement of a partially time-expanded network. In such a network, only those parts are fully discretized where optimality cannot be guaranteed based on a coarser degree of discretization.

Within the area of Service Network Design, \cite{scherr2020dynamic} extended the model to the Service Network Design problem with Mixed Autonomous Fleets by adding additional properties to the partially time-expanded network, allowing themselves to manage the flow of vehicles more precisely. Moreover, they improve upon the algorithm by strengthening lower bounds through the use of valid inequalities and by introducing heuristics methods to find strong upper bounds more easily. \cite{marshall2021interval} introduced an Interval-based DDD approach where the nodes in the time-expanded network represent a certain location within a time interval instead of at a specific time. The resulting solution structure is exploited in the refinement stage of the algorithm, leading to fewer iterations and smaller partially time-expanded networks. Consequently, larger instances can be solved and high-quality solution are found more quickly than in the original DDD.

The application of DDD has been extended to other types of transportation problems as well. \cite{hewitt2019enhanced} adapted and enhanced the algorithm for the Service Network Design problem to fit its less-than-truckload freight transportation analogue: the Continuous-Time Load Plan Design Problem. He proposed a two-phase algorithm that solves a linear relaxation in the first phase, after which a set of arc-time window inequalities enables the definition of feasible time windows for the transportation of commodities. \cite{vu2020dynamic} solved the Time Dependent Traveling Salesman Problem with Time Windows by means of DDD. They showed that it outperforms even the strongest algorithms in existing literature and that it is robust to all instance parameters. This work has been extended in \cite{vu2022solving} for solving the Time Dependent Minimum Tour Duration Problem and the Time Dependent Delivery Man Problem. \cite{he2021dynamic} applied DDD to the Minimum Duration Time-Dependent Shortest Path problem with piecewise linear travel times and find that the number of explored breakpoints drastically reduces, making the algorithm highly efficient and scalable. \cite{lagos2022dynamic} successfully developed and tested a DDD algorithm to solve the Continuous Time Inventory Routing Problem with Out-and-Back Routes. 
\section{Multi-Depot Vehicle Scheduling Problem with Trip Shifting}
\label{s:mdvsp}
The Multi-Depot Vehicle Scheduling Problem considers a set of scheduled trips $M$, where each trip $m \in M$ departs from a location $l^s_m$ at time $t^s_m$ and arrives at a location $l^e_m$ at time $t^e_m$. We let $\tau_m:= t^e_m-t^s_m$ denote the trip time of trip $m$. The aim is to optimally route vehicles through a network such that all trips are covered. Every single vehicle route must start its route at the beginning of the time horizon at one of the depots $d \in D$ and is obligated to return to the exact same depot before the end of the time horizon. Through the introduction of trip shifting to the MDVSP, trips are allowed to depart at most $\delta^{\max}$ earlier or later than originally timetabled.

\subsection{Time-Space Network Representation}
The network consists of a separate layer for each depot $d \in D$. The full network $G = (N, A)$ is thus the union of $|D|$ individual network layers $G^d = (N^d, A^d)$, where $N^d$ and $A^d$ are the set of nodes and arcs in the network layer for depot $d \in D$, respectively. A node $(l, t)$ represents a location $l \in L$, which is either a station or a depot, at a certain point in time $t$. Every layer contains two nodes for the depot: $i_d^s$ and $i_d^e$, representing the depot at the start and end of the horizon, respectively. The travel time between two locations $l, k \in L$ is denoted by $\tau_{lk}$. We assume that all trip and travel times are integer. As a result, it suffices to consider shifting departure times by an integer number of time units to find a continuous-time optimal solution. 

The nodes in a network layer are connected through several types of arcs, namely:

\begin{itemize}
	\item \textbf{Trip arcs} represent timetabled trips $m \in M$ and their possible deviations, and connect node $(l_m^s, t_m^s + \delta)$ corresponding to the start location and start time of the trip with node \\$(l_m^e, t_m^e+\delta)$ corresponding to the end location and end time, for $\delta \in \{-\delta^{\max}, -\delta^{\max} + 1, \dots,$ $\delta^{\max} - 1, \delta^{\max}\}$.  The set of all trip arcs belonging to trip $m\in M$ in the layer for depot $d\in D$ is denoted as $A^d(m)$. In order to cover a trip, either one of its trip arcs must be traversed.
	\item \textbf{Deadheading arcs} represent empty moves of vehicles between stations and connect the end station of one trip with the start station of another.
	\item \textbf{Pull-out and pull-in arcs} represent empty moves of vehicles in and out of the depot and connect the depot to the start station of a trip in case of a pull-out, or from the end station of a trip back to the depot in case of a pull-in. 
	\item \textbf{Waiting arcs} represent vehicles standing still and waiting at the depot or a station and connect every node at a location to the consecutive node of that location in time.
\end{itemize}

\noindent The foundation of the mathematical model for the MDVSP-TS is based on the multi-commodity time-space network flow formulation proposed by \cite{kliewer2006time} for the MDVSP. As this formulation does not consider trip shifting, we incorporate this feature in a similar fashion as \cite{kliewer2012multiple} in their model for the integrated vehicle and crew scheduling problem with multiple depots and consideration of time windows for scheduled trips.

Let $c_{ij}^d$ denote the costs of traversing arc $(i, j) \in A^d$ belonging to depot $d \in D$. Next, we introduce the decision variables $x_{ij}^d$ representing the number of vehicles traversing arc $(i, j) \in A^d$ belonging to depot $d \in D$. Altogether, the MDVSP-TS can be formulated as follows:
\begin{mini!}
	%
	{}
	%
	{\sum_{d \in D}\sum_{(i,j) \in A^d} c_{ij}^d x_{ij}^d, \label{eq:mdvspObj}}
	%
	{\label{formulation:mdvsp}}
	%
	{}
	%
	\addConstraint
	{\sum_{j:(j,i) \in A^d}x_{ji}^d - \sum_{j:(i,j) \in A^d}x_{ij}^d}
	{= 0 }
	{\forall d \in D, \forall i \in N^d \setminus \{i_d^s, i_d^e\},                \label{eq:flow}}
	\addConstraint
	{\sum_{d \in D}\sum_{(i,j) \in A^d(m)} x_{ij}^d}
	{= 1}
	{ \forall m \in M,               \label{eq:cover}}
	\addConstraint
	{ x_{ij}^d }
	{\in \mathbb{Z}_{\geq 0} \quad \label{eq:dom}}
	{\forall d \in D,\forall (i, j) \in A^d.}
\end{mini!}%
The objective function (\ref{eq:mdvspObj}) minimizes the total operating cost of the vehicles in the network. Flow conservation is ensured by constraints (\ref{eq:flow}), whereas constraints (\ref{eq:cover}) make sure that every trip $m$ is covered by exactly one vehicle by using one of the trip arcs in $\bigcup_{d \in D}A^d(m)$.

Throughout this paper, we assume that the objective is to minimize a combination of the number of vehicles and the total driving distance. In other words, there are no waiting costs and deviating from the original timetable is not penalized. We also assume that distances satisfy the triangle inequality. 

\subsection{Deadheading Arc Aggregation}
\label{ss:aggregation}
Key to the ability of the time-space representation to perform well on large instances is the aggregation of deadheading arcs in the network \citep{kliewer2006time}. Instead of explicitly modeling any possible connection between two trips at any point in time, arcs that are redundant are removed. By aggregating all arcs that enable two trips to be executed consecutively, the size of the network and the corresponding number of variables in the mathematical model is heavily reduced without compromising the feasible region of solutions. Aggregation is performed by means of the following two-step procedure:

 \textbf{First stage aggregation.} 
For each timetabled trip $m$ arriving at location $k$, the earliest compatible trip is considered at each location $l$ other than $k$. Then, a deadheading arc is constructed to the so-called \textit{first-match} of trip $m$ with location $l$, departing from the end node of trip $m$ and arriving at the starting node of the first-match. All later matches of trip $m$ are disregarded, as they do not enable new connections to be made. Trips at location $l$, $l \neq k$ that are not first-matches can be reached by consecutively traversing the first-match arc to $l$ and the waiting arcs at $l$.

 \textbf{Second stage aggregation.}
The set of first-match deadheading arcs is further reduced in the second stage of the aggregation process. Consider a trip $m$ at location $k$ and let $F_l$ be the set of trips arriving at any location $l$, $l \neq k$ that consider trip $m$ as their first-match at location $k$. Out of set $F_l$, only the trip with the latest arrival time keeps its direct deadheading connection with trip $m$, which we call the \textit{latest-first-match}. All the other arcs are removed from the model as their connection can be replaced by traversing waiting arcs at location $l$ combined with the latest-first-match arc.

\begin{figure}[H]
	\centering
	\includegraphics[width=\textwidth]{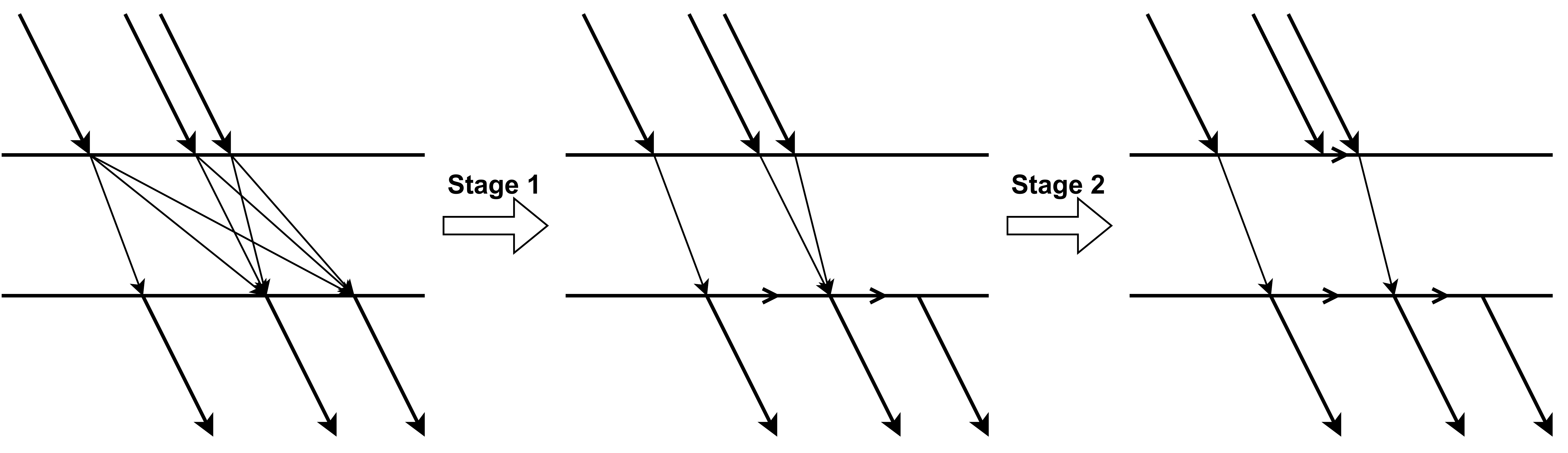}
	\caption{Network reduction through the two-stage aggregation process.}
	\label{fig:aggregation}
\end{figure}

\noindent Figure \ref{fig:aggregation} shows how deadheading arcs are aggregated between two locations in the network. At first, all feasible deadheading possibilities from the location with incoming trips to the location with outgoing trips are modeled explicitly. After the first stage, only deadheading arcs to the \textit{first-matches} of each of these incoming trips remain. In the second stage of aggregation, we switch our attention to the outgoing trips. Whenever an outgoing trip has multiple incoming \textit{first-match} arcs from the same location, all but the \textit{latest-first-match} are deleted.
\section{Dynamic Discretization Discovery}
\label{s:ddd}
DDD is an iterative refinement algorithm introduced by \cite{boland2017continuous} that uses a partially time-expanded network as a starting point. Algorithm \ref{alg:ddd} gives a high-level overview of the procedure. The partially time-expanded network $G_T$ serves as an optimistic representation of the problem. It does not explicitly model each potential trip shift, but underestimates the length of arcs to decrease the number of nodes and arcs in the network. In each iteration, the problem is solved on the relaxed network, yielding a lower bound to the optimal objective value. If the solution is \textit{implementable}, that is, can be converted to a feasible solution on the full network with the same cost, an optimal solution for the MDVSP-TS has been obtained. Otherwise, the partially time-expanded network is refined and a new iteration of the algorithm is entered. As long as any continuous-time feasible sequence of trips is represented in the partial network,  Algorithm~\ref{alg:ddd} can be used to solve the MDVSP-TS to any desired precision $\epsilon$.

Sections~\ref{ss:partial} - \ref{ss:adaptive} dive deeper into the implementation of each of the steps within the algorithm. Section~\ref{ss:pp} describes how the solution can be post-processed to eliminate unnecessary alterations from the timetable. 

\begin{algorithm}[H]
	\caption{Dynamic Discretization Discovery}
	\label{alg:ddd}
	\begin{algorithmic}[1]
		\State Create an initial partially time-expanded network $G_T$ that induces a relaxation of the MDVSP-TS
		\While{not solved}
		\State Solve MDVSP-TS($G_T$)
		\State Convert the solution on the partial network $G_T$ to a feasible solution on the full network
		\State network $G$
		\If {optimality gap $< \epsilon\%$}
		\State Terminate
		\EndIf
		\State Refine $G_T$ by lengthening one or more understimated arcs
		\EndWhile
	\end{algorithmic}
\end{algorithm}

\subsection{Creating a partially time-expanded network $G_T$}
\label{ss:partial}
Within the DDD framework, the partially time-expanded network should be constructed such that it can be regarded as a relaxation of the full network, generating valid lower bounds. We describe the network for a single network layer, omitting the index $d$ for brevity. For the MDVSP-TS, the partially time-expanded network is based on a node set $N_T$, which contains a subset of the nodes in the fully discretized network. To ensure that there always exists a feasible solution, we initialize the partially time-expanded network by including for every trip $m$ the start node $(l^s_m,t_m^s-\delta^{\max})$ and the end node $(l^e_m,t_m^e-\delta^{\max}).$

 As in the fully discretized network, the arc set $A_T$ contains trip arcs, deadheading arcs, pull-out and pull-in arcs and waiting arcs. The pull-out, pull-in and waiting arcs are constructed in the same way as in the full network. A pull-out arc connects each starting depot node to the first node at each location and a pull-in arc connects the last node at each location to the ending depot node. Moreover, waiting arcs are created for every location to connect nodes at the same location or depot throughout the time horizon. The trip and deadheading arcs are constructed such that any feasible route in the full network is represented in the partial network, such that MDVSP-TS($G_T$) is indeed a relaxation of MDVSP-TS defined on the full network. 
 
To define the trip arcs, let $\rho_l(t)$ denote the latest time point at or before $t$ at location $l$ which corresponds to a node in $N_T$. From every node $(l^s_m,t_m^s+\delta)$ corresponding to the start of a trip, we add a trip arc to node $\left(l_m^e,\rho_{l_m^e}\left(t_m^e+\delta\right)\right)$. Note that since we required that the network contains start and end nodes corresponding to a shift of $-\delta^{\max}$ for each trip, it holds that $ t_m^e-\delta^{\max}\leq\rho_{l_m^e}\left(t_m^e+\delta\right) \leq t_m^e+\delta^{\max}$. Figure~\ref{fig:tripArcsPartial} illustrates this construction for a trip $m$ with $t_m^s=3$, $\tau_m=2$, $\delta^{\max}=3$. As not all time points are present in the discretization, the trip arcs starting from $(l_m^s,3)$ and $(l_m^s,6)$ are ``rounded down" to a time point that is included. 
\begin{figure}[hb]
	\centering
	\includegraphics{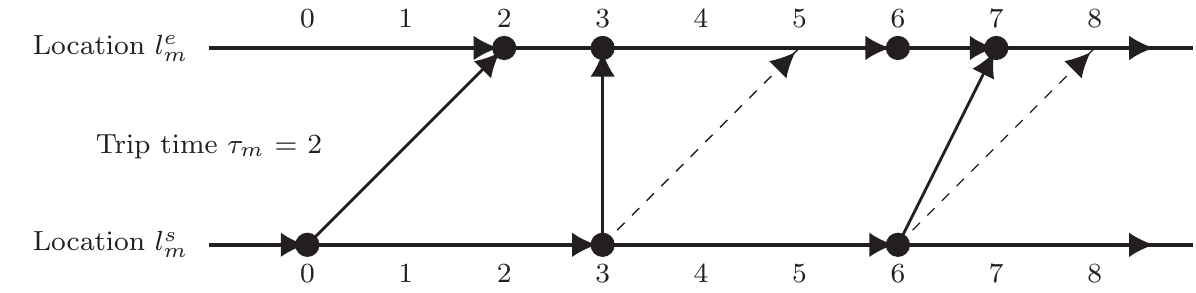}
	\caption{Trip arcs in the partially time-expanded network for a trip $m$ with $t_m^s=3$, $\tau_m=2$ and $\delta^{\max}=3$. Dashed arcs indicate the correct trip times, solid arcs are included in the partial network.}
	\label{fig:tripArcsPartial}
\end{figure}

The deadheading arcs connect every node $(k,t)$ that is the head of at least one trip arc to each location $l$ that serves as the origin of at least one trip. To ensure that the partially time-expanded network induces a valid relaxation of the continuous-time problem, DDD typically requires that all arcs $\left((l,t),(k,t')\right)$ satisfy the property that $t'\leq t+\tau_{lk}$, i.e.\ that all arcs are either the correct length or ``short". However, for the MDVSP-TS, we can identify cases where we can include ``long" arcs without destroying any continuous-time feasible solution. To assess the effectiveness of these techniques, we propose three variants for the deadheading arcs construction. 

\textbf{Short.} In this scheme, we add deadheading arcs between $(k,t)$ and $(l,\rho_l(t+\tau_{kl}))$. That is, arcs are always either the correct length or too short, as is customary in DDD. This is illustrated in Figure~\ref{fig:dh1}. 

\textbf{Medium.} Here, we exploit that we do not destroy any continuous-time feasible connections if we disregard arcs that underestimate the true travel time by more than twice the maximum deviation. In any solution, the start time of the trip that corresponds with node $(l,\rho_l(t+\tau_{kl}))$ is at most $\rho_l(t+\tau_{kl})+\delta^{\max}$. Therefore, if $\rho_l(t+\tau_{kl})<t+\tau_{kl}-2\delta^{\max}$, the trips corresponding to $(k,t)$ and $(l,\rho_l(t+\tau_{kl}))$ can never be performed in sequence. In this case, we can therefore ``round up" the arc, instead of round down. This construction is illustrated in Figure~\ref{fig:dh2}. The arc starting at $(k,1)$ is now rounded up instead of down, since otherwise the true travel time is underestimated by more than $2$. 

\textbf{Long.} In this scheme, we explicitly check to which trip (or trips) node $(l,\rho_l(t+\tau_{kl}))$ corresponds, and only add the connection if this trip is compatible with node $(k,t)$, i.e. if $t+\tau_{kl}$ falls inside its time window. If not, we round up the arc. This is illustrated in Figure~\ref{fig:dh3}. The arc starting at $(k,5)$ is now rounded up instead of down, as from this node there does not exists a starting time of trip $m_l(7)$ that can be reached in time.


We note that both for trip arcs and deadheading arcs, not all constructed arcs may actually enable new connections in the network. Hence, we remove redundant arcs according to the same aggregation procedure described in Section~\ref{ss:aggregation}. 

\begin{figure}[ht]
     \centering
     \begin{subfigure}[b]{0.9\textwidth}
         \centering
         \includegraphics[width=\textwidth]{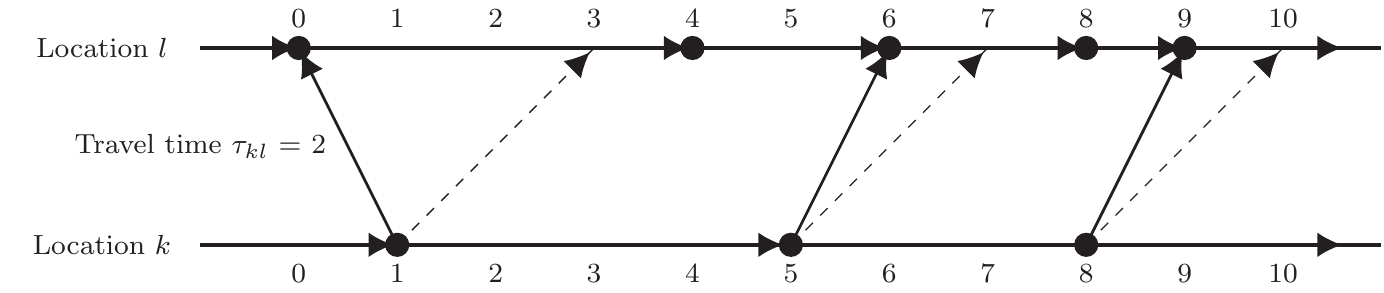}
         \caption{Short.}
         \vspace{20pt}
         \label{fig:dh1}
     \end{subfigure}
     \begin{subfigure}[b]{0.9\textwidth}
         \centering
         \includegraphics[width=\textwidth]{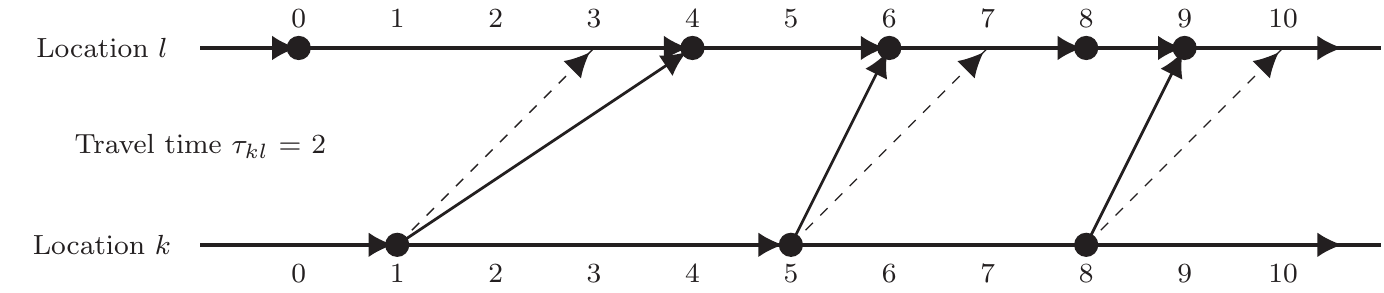}
         \caption{Medium.}
         \vspace{20pt}
         \label{fig:dh2}
     \end{subfigure}
     \begin{subfigure}[b]{0.9\textwidth}
         \centering
         \includegraphics[width=\textwidth]{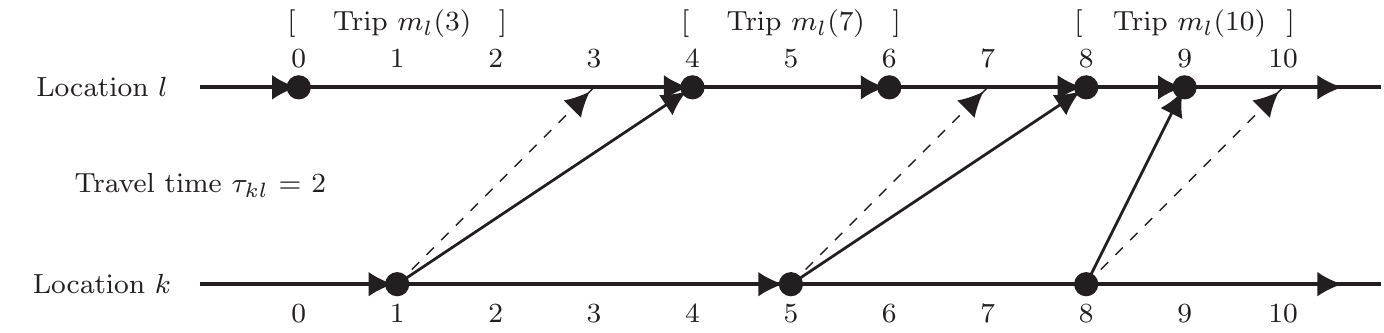}
         \caption{Long.}
         \label{fig:dh3}
     \end{subfigure}
        \caption{Deadheading arcs in the partially time-expanded network, with $\delta^{\max}=1$ and three trips originating at location $l$. Dashed arcs indicate the correct travel times, solid arcs are included in the partial network.}
        \label{fig:three graphs}
\end{figure}


\subsection{Refinement Strategies}
\label{ss:convert}

The solution to the MDVSP-TS on the partially time-expanded network does not guarantee feasibility on the full network, as the solution may contain one or more arcs with underestimated length. Such arcs may connect trips that cannot be executed consecutively when actual travel times are considered. However, the presence of underestimated arcs does not automatically deem a schedule infeasible. For example, if an arc that is 5 minutes too short is preceded by a waiting arc of at least that length, a feasible solution could be constructed by simply departing earlier. Next, we present three refinement strategies that evaluate whether a solution on the partially time-expanded network can be converted to a continuous-time feasible solution and if not, suggest one or more time points that should be added to the discretization. 

\subsubsection{Duty Refinement}

The first refinement strategy decomposes the solution of the MDVSP-TS on the partial network into duties, the sequences of trips that are carried out by the vehicles. Then, it checks whether each duty can be converted into a continuous-time feasible route, by picking the earliest possible departure time for each trip that respects actual travel times. If not all duties are implementable, it adds time points to the discretization to ensure that infeasible sequences of trips are no longer connected in the partial network in future iterations of DDD.

Formally, let $p$ denote the sequence of trips of a specific vehicle in visiting order in the current solution on the partially time-expanded network. The departure time of trip $i \in p$ is denoted by $\pi_i$. With slight abuse of notation, let $\tau_{i,i+1}$ denote the sum of the trip time of trip $i$ and the deadheading time between the end location of trip $i$ and the start location of trip $i+1$. Commencing at the first trip at $\pi_1 = t_1^s-\delta^{\text{max}}$, the departure time of trip $i$ is computed as $\pi_i = \text{max}\{\pi_{i-1} + \tau_{i-1,i}, t^s_{i} - \delta^{\max}\}$. If $\pi_i > t^s_{i} + \delta^{\max}$ for any of the trips, the vehicle route is infeasible. 

In case all duties are feasible, we have found the optimal solution and DDD terminates. Otherwise, the infeasible duties are processed one by one and time points are added to the partially discretized network to ensure that these duties are no longer present in the network in further iterations of DDD. For a given infeasible duty $p$, it holds that there exists some bottleneck trip $j^*$ where $\pi_{j*}>t^s_{j^*}+\delta^{\max}$. Then, for each $i<j^*$, the duty refinement strategy adds the nodes $(l_{i}^s,\pi_i)$ and $(l_{i}^e,\pi_i+\tau_i)$ to the discretization, i.e. a node at the start location at the correct start time and a node at the end location at the correct end time. For the bottleneck trip itself, it suffices to add the node $(l_{j^*}^s,\pi_{j^*})$. 

\begin{prop}
Let $p$ be a non-implementable duty. Then, after applying Duty Refinement, $p$ cannot be represented in the partially time-expanded network $G_T$. 
\label{prop}
\end{prop}
\begin{proof}
After duty refinement, $G_T$ contains the nodes $(l_{i}^s,\pi_i)$ and $(l_{i}^e,\pi_i+\tau_i)$ for $i<j^*$ and $(l_{j^*}^s,\pi_{j^*})$. It follows that in $G_T$, we can connect the trips 1 up to $j^*-1$ using the path $$(l_{1}^s,\pi_1)\rightarrow (l_{1}^e,\pi_i+\tau_1)\rightarrow(l_{2}^s,\pi_2)\rightarrow (l_{2}^e,\pi_2+\tau_2)\rightarrow \hdots \rightarrow(l_{j^*-1}^s,\pi_{j^*-1})\rightarrow (l_{j^*-1}^e,\pi_{j^*-1}+\tau_{j^*-1}).$$
Moreover, since this path takes the earliest possible departure options for each trip, there does not exists a path that contains the same trips but arrives at $l_{j^*-1}^e$ earlier than $\pi_{j^*-1}+\tau_{j^*-1}$. From this node, it follows from the definition of $\pi$ that the deadhead arc to the start location of trip $j^*$ connects to the node $(l_{j^*}^s,\pi_{j^*})$. However, as $p$ is non-implementable, it holds that $\pi_{j*}>t^s_{j^*}+\delta^{\max}$, such that this node does not correspond to a start node of a trip arc belonging to $j^*$. Therefore, it is impossible to connect all trips in $p$ in a single path in $G_T$. 
\end{proof}
It directly follows from Proposition~\ref{prop} that the proposed scheme returns a continuous-time optimal solution in a finite number of iterations, since there only exists a finite number of duties. 

\begin{figure}[ht]
     \centering
     \begin{subfigure}[b]{0.48\textwidth}
         \centering
         \includegraphics[width=\textwidth]{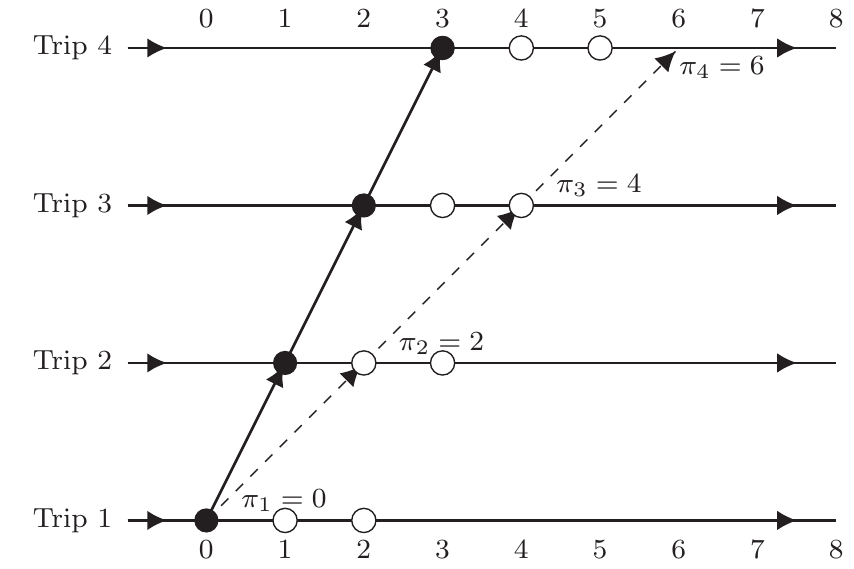}
         \caption{Before refinement.}
         \label{fig:rs1}
     \end{subfigure}
     \hfill
     \begin{subfigure}[b]{0.48\textwidth}
         \centering
         \includegraphics[width=\textwidth]{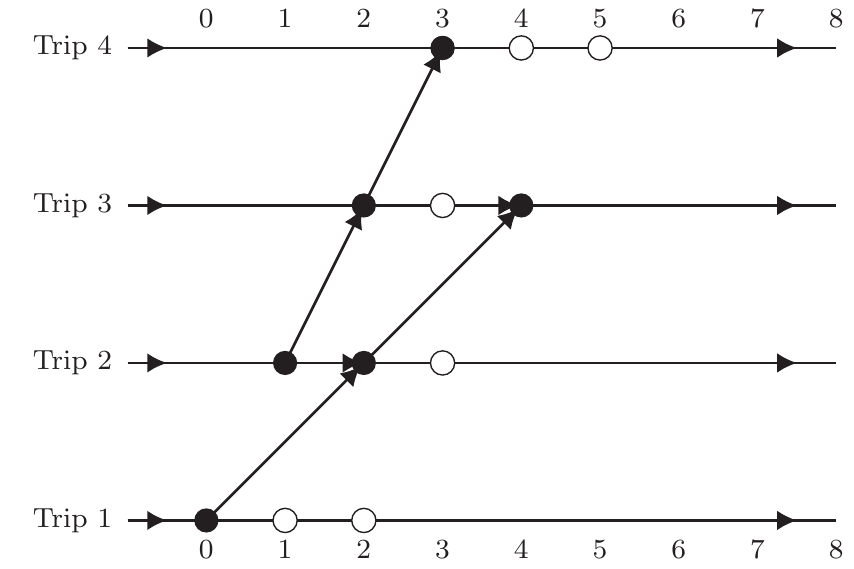}
         \caption{After refinement.}
         \label{fig:rs2}
     \end{subfigure}
        \caption{Illustration of duty refinement for a duty with four trips with original start times at 1, 2, 3 and 4 respectively, and $\delta^{\max}=1$.}
        \label{fig:rs}
\end{figure}

Figure~\ref{fig:rs} presents a visualization of the duty refinement strategy for an infeasible duty containing four trips, where $\tau_{i,i+1}=2$ for $i=1, 2,$ and 3, and $\delta^{\max}=1$. Note that the depicted network is not the partially time-expanded network itself, but a derivative; the solid nodes correspond to start nodes of the trips that are present in the partial network; the solid arcs correspond to connections that are possible in the partial network. The dashed arcs indicate the actual travel times. It can be seen in Figure~\ref{fig:rs1} that when evaluating the feasibility of this duty, we find that the earliest possible starting time of trip 4 is at $t=6$, which violates the maximum deviation. Figure~\ref{fig:rs2} depicts the situation after refinement, where it is no longer possible to connect trips 1 up to 4 due to the presence of additional nodes.

\subsubsection{Optimizing the Duty Decomposition}
\label{sec:optdd}
A disadvantage of the duty refinement strategy is that its outcome may depend on the used duty decomposition. To illustrate the impact of selecting the duty decomposition, consider Figure~\ref{fig:whyDecomp}. In this example, the trip sequences (1, 3, 5), (1, 3, 6) and (2, 4, 6) are feasible. However, the sequence (2, 4, 5) is infeasible, as the earliest departure times for these trips will be 1, 3 and 5, respectively, but the latest allowed departure time of trip 5 is 4. Therefore, the duty decomposition (1, 3, 6), (2, 4, 5) is infeasible, whereas the decomposition (1, 3, 5), (2, 4, 6) is feasible.  We next develop two refinement strategies that exploit this feature by \textit{optimizing} the duty decomposition.  Intuitively, by carefully distributing the used short arcs over the duties we have a higher chance of ending up with implementable duties. 

\begin{figure}[h]
\centering \includegraphics{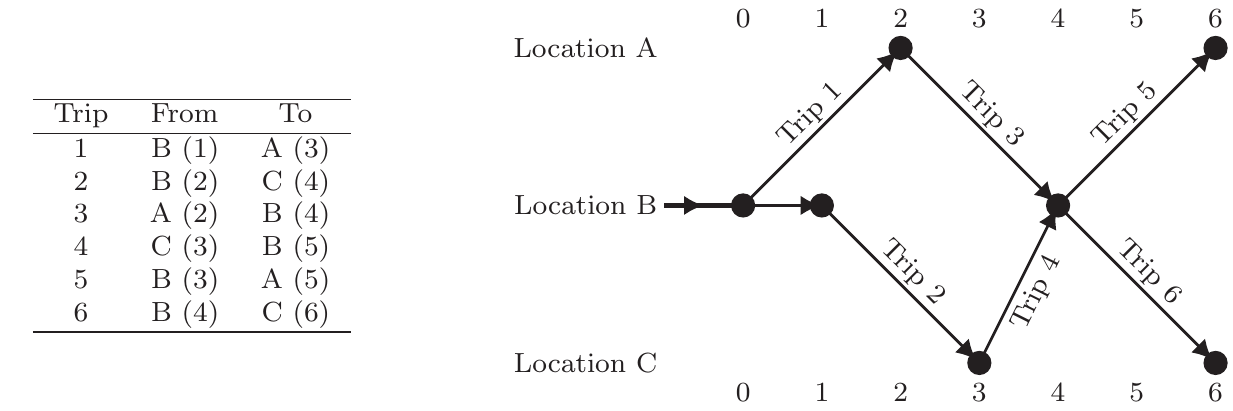}
    \caption{Solution on the partial network with a non-unique duty decomposition. The table gives the original timetable, and $\delta^{\max}=1$.}
    \label{fig:whyDecomp}
\end{figure}

As a first step, we decompose the solution in the partial network into components. Depot nodes are disregarded when determining these components, so a pair of nodes belongs to the same component if they are connected by a path that does not contain any pull-out or pull-in arcs. This decomposition is visualized in Figure~\ref{fig:comp}.

\begin{figure}[h]
\centering \includegraphics{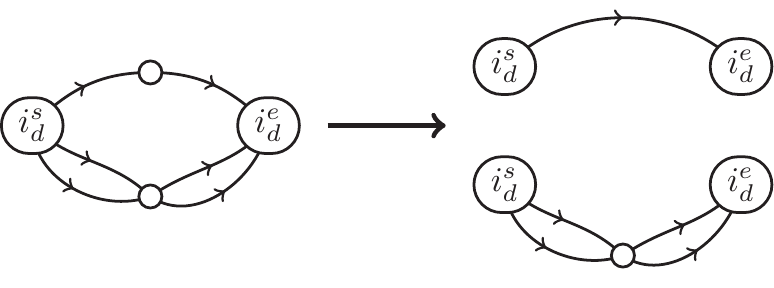}
    \caption{Illustration of the decomposition of a solution into components.}
    \label{fig:comp}
\end{figure}

Next, we examine the component $C$. Let $P$ denote the set of all duties that are supported by a component and let $M_C$ denote the set of all trips that are performed by duties in $P$. Let $P_m$ denote the set of duties that cover trip $m$. By evaluating the implementability of all $p\in P$, we can determine $n_p$, the number of time points that needs to be added to $G_T$ were we to select $p$ in the duty decomposition. After introducing binary decision variables $z_p$ for selecting duties in the decomposition, we can determine the continuous-time feasibility of the component $C$ by solving the following binary program:
\begin{mini!}
	%
	{}
	%
	{\sum_{p\in P}n_pz_p, \label{eq:decompOptObj}}
	%
	{\label{formulation:decomp}}
	%
	{}
	%
	\addConstraint
	{\sum_{p\in P_m} z_p}
	{= 1 \label{eq:coverDecomp}}
	{\forall m \in M_C,}
	\addConstraint
	{z_p}
	{\in \mathbb{B} \quad \label{eq:decompDom}}
	{\forall p \in P.}
\end{mini!}%
The objective~(\ref{eq:decompOptObj}) minimizes the number of time points that needs to be added to the partial network. Constraints~(\ref{eq:coverDecomp}) ensure that the duty decomposition is valid by covering every trip. 

If Problem~(\ref{formulation:decomp}) returns an objective value of 0, the component is feasible. Otherwise, the selected duty decomposition contains at least one non-implementable duty and it is necessary to further refine the partial network. We incorporate this idea in two refinement strategies. The first strategy, \textit{Fewer Time Points}, applies duty refinement to the infeasible duties in the optimal solution to Problem~(\ref{formulation:decomp}), ensuring that the number of added time points in an iteration is minimized. The second strategy, \textit{Fewer Iterations}, is more aggressive and applies duty refinement to \textit{all} infeasible duties in $P$, even those that were not selected in the duty decomposition. Here, the intuition is that these infeasible duties might otherwise show up in later iterations of DDD, which we can prevent by refining them as a precaution. 

\subsection{Computing Feasible Solutions}
Whenever a trip in a duty in the partially time-expanded network solution cannot be executed respecting actual travel time and time window constraints, the solution as a whole is infeasible. However, we can often convert it to a feasible solution relatively quickly, by maintaining the feasible part and re-optimizing the infeasible part. To do so, we simply solve the MDVSP-TS with the trips that are not covered by feasible duties on a fully discretized network. Based on preliminary experiments, we only call this procedure when the number of trips performed by infeasible duties is at most 100, such that the time spent on computing upper bounds is very small compared to the whole algorithm.

\subsection{Adaptive MIP Tolerance}
\label{ss:adaptive}
In early iterations of DDD, the partially time-expanded network typically induces a relatively weak relaxation of the problem, such that there is a large gap between the lower and upper bound. Therefore, it pays off to start with a larger optimality tolerance and tighten the tolerance as the algorithm progresses. Based on preliminary experiments, we initially set a \textit{relative} optimality tolerance of 1\%. Once DDD has found a feasible solution, we set an \textit{absolute} optimality tolerance of $\frac{\text{UB}-\text{LB}}{10}$, where $\text{UB}$ and $\text{LB}$ are the current upper and lower bound, respectively.

\subsection{Post-Processing} 
\label{ss:pp}
As deviations from the original timetable are not penalized, the solution returned by DDD might alter the timetable unnecessarily. Therefore, in a post-processing step, we minimize these deviations, keeping the costs of the solution fixed. Analogous to the refinement strategies, we present two post-processing strategies. 

\subsubsection{Duty Post-Processing} 

A new mathematical model is solved for every duty that minimizes deviations. Given duty $p$, this model chooses departure times for all trips $i\in p$ such that all trips on the route can be executed within the prespecified time window whilst adhering to the actual travel times between locations. 

Alongside the departure time variable $\pi_i$, decision variables $\delta^+_i$ and $\delta^-_i$ are introduced, representing the potential delay or advancement of the departure, respectively. Both $\delta^+_i$ and $\delta^-_i$ are integer-valued and range from 0 to $\delta^{\max}$. This results in the following formulation:

\begin{mini!}
	%
	{}
	%
	{\sum_{i \in p}\delta^+_i + \delta^-_i, \label{eq:ppObj}}
	%
	{\label{formulation:pp}}
	%
	{}
	%
	\addConstraint
	{\pi_i+\tau_{i,i+1}}
	{\leq \pi_{i+1} }
	{\forall i \in p \setminus |p|,               \label{eq:compat}}
	\addConstraint
	{\pi_i}
	{ =t_i^s+\delta^+_i - \delta^-_i \quad}
	{\forall i \in p,               \label{eq:linkPiDelta}}
	\addConstraint
	{\delta^+_i,\delta^-_i}
	{\in \left\{0,1,...,\delta^{\max} \right\} \quad \label{eq:deltaDom}}
	{\forall i \in p,}
	\addConstraint
	{\pi_i}
	{\geq 0 \quad \label{eq:piDom}}
	{\forall i \in p.}
\end{mini!}%
The objective function (\ref{eq:ppObj}) minimizes the sum of deviations from the timetable. For every set of consecutive locations on the vehicle duty, compatibility is protected by the travel time constraints (\ref{eq:compat}). If trip $j$ is performed after trip $i$, the departure time of $j$ has to be larger than or equal to the sum of the departure time of $i$, the trip time of $i$ and the deadheading time between $i$ and $j$. Constraints~(\ref{eq:linkPiDelta}) set $\pi_i$ equal to the sum of the original departure time of the trip ($t^s_{i}$), the delay $\delta^+_i$ and the advancement -$\delta^-_i$.

\subsubsection{Optimizing the Duty Decomposition} 
In Section~\ref{sec:optdd}, we optimized the duty decomposition with respect to the number of time points that should be added in the refinement step. In a similar fashion, it is also possible to optimize the duty decomposition to minimize deviations from the timetable. This can be achieved by replacing the objective function in Problem~(\ref{formulation:decomp}) by 
$$\sum_{p\in P}\delta_pz_p,$$
where $\delta_p$ denotes the minimum deviation possible by duty $d$, which is obtained by solving Problem~(\ref{formulation:pp}) for every $p\in P$.

\section{Computational Results}
\label{s:results}
This section dives into the computational results of the DDD algorithm. First, a description is given of the data generating process and the resulting problem instances, followed by a detailed analysis of the performance of DDD.

\subsection{Problem Instances}
Datasets for the MDVSP will be generated according to the data generating procedure introduced by \cite{carpaneto1989branch} and later employed by, amongst others, \cite{pepin2009comparison}, \cite{desfontaines2018multiple} and \cite{kulkarni2018new}. These problem instances contain $|L|$ locations that serve as departure or arriving locations for trips, spread out randomly over a square grid, and four depots, located in the corners of the grid. There are two types of trips.

\begin{itemize}
	\item \textbf{Short trips} start and end at different, randomly chosen locations. The execution time of a short trip exceeds the travel time between two locations by at most 45 minutes. Starting times during peak hours (7:00 - 8:00 and 17:00 - 18:00) are chosen with a higher probability than off-peak departures. In the default setting, short trips comprise 40\% of the total number of trips $|M|$.
	\item \textbf{Long trips} are round trips departing and arriving at the same, randomly chosen location. They are distributed randomly over the time horizon and take between three and five hours to complete.
\end{itemize}

\noindent The cost $c_{ij}$ of traversing an arc $(i, j) \in A$ in the network is calculated as follows:

\begin{itemize}
	\item If $(i, j)$ is a trip arc, $c_{ij} = 0$.
	\item If $(i, j)$ is a deadheading arc. $c_{ij} = \tau_{ij}$.
	\item If $(i, j)$ is a waiting arc, $c_{ij} = 0$.
	\item If $(i, j)$ is a pull-in or pull-out arc, $c_{ij} = 500 + \tau_{ij}$.
\end{itemize}

\noindent We consider instances with 500, 750 and 1000 trips, consisting of a mix (`M') of 40\% short trips and 60\% long trips. For each problem type, ten different instances are generated. By default, the number of locations in the network is one tenth of the number of trips. The experiments are performed on a machine with an AMD Rome 7H12 2.6GHz
processor with 32 cores and 64 GB RAM. CPLEX 22.1 is used to solve the mixed-integer programs. We use the default solver settings, except for the optimality tolerance, which we set adaptively according to Section~\ref{ss:adaptive}. We terminate DDD if the absolute gap between lower and upper bound is strictly smaller than 1, i.e. once we have proven exact optimality. 


\subsection{Impact of Deadheading Arc Construction} 
In the first experiment, we fix the refinement strategy to Duty Refinement, and solve the 500M instances with the Short, Medium and Long deadheading arc generation schemes. As a benchmark, we solve the instances on the fully discretized network. We consider a maximum deviation $(\delta^{\max})$ of 1 up to 5 minutes. 

Figure~\ref{fig:compareDH} presents the average computation time for the three deadheading arc construction techniques and the fully discretized model. DDD with Short deadheading arcs performs considerably worse than Medium and Long, and only outperforms the benchmark for the largest value of the maximum deviation. With Medium and Long deadheading arcs, DDD is not only faster than the fully discretized model for every value of $\delta^{\max}$, but its relative speed increases in the value of $\delta^{\max}$. In the most complex case of a 5-minute maximum deviation, naively modeling all potential trip shifts explicitly leads to 11 times the running time of DDD with Long deadheading arcs.

We can explain the difference in performance of the different deadheading arc construction schemes using Table~\ref{tbl:compareDAG}, which presents the average number of iterations, number of nodes in the partial network in the final DDD iterations and the obtained lower bound in the first iteration. Compared to Medium and Long deadheading arcs, DDD with Short deadheading arcs produces a first lower bound that is much weaker and requires many more iterations and nodes to prove optimality. DDD with Long arcs also clearly outperforms DDD with Medium arcs, be it with by a smaller margin. What is surprising is that the first bound obtained with Medium or Long deadheading arcs is exceptionally strong: for $\delta^{\max}$ equal to 5 minutes, the absolute gaps between the first lower bound and optimal solution are only 340 (0.27\%) and 53 (0.04\%), respectively.

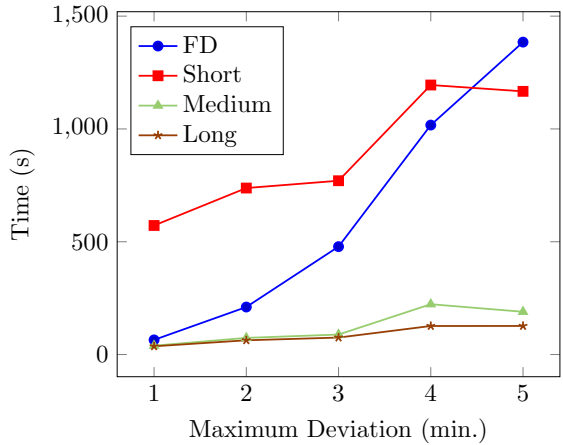
\begin{figure}[ht]
\centering
\begin{tikzpicture}[scale=0.85]
    \centering
\begin{axis}[xlabel={Maximum Deviation (min.)},
  ylabel={Time (s)},every axis plot/.append style={thick}, legend style={at={(0.03,0.96)},anchor=north west},legend cell align={left}]
\addplot table [x=DeltaMax, y=Full, col sep=comma] {compareDH.csv};
\addlegendentry{FD}
\addplot+[mark=square*,mark options={fill=red}] table [x=DeltaMax, y=Short, col sep=comma] {compareDH.csv};
\addlegendentry{Short}
\addplot+[YellowGreen,mark=triangle*,mark options={fill=YellowGreen}] table [x=DeltaMax, y=Medium, col sep=comma] {compareDH.csv};
\addlegendentry{Medium}
\addplot+[RawSienna	,mark=star,mark options={fill=RawSienna	}] table [x=DeltaMax, y=Long, col sep=comma] {compareDH.csv};
\addlegendentry{Long}
\end{axis}
\end{tikzpicture}
\caption{Average time to optimality on the 500M instances for the fully discretized model (FD) and for DDD with Short, Medium and Long deadheading arc construction.}
\label{fig:compareDH}
\end{figure}

\begin{table}[b]
\caption{Computational results for the 500M instances with different deadheading arc construction schemes. For each value of the maximum deviation, the objective is reported, as well as for each scheme the average number of iterations, number of nodes in the final iteration of DDD and the first lower bound.}
\label{tbl:compareDAG}
	\resizebox{\columnwidth}{!}{%
\begin{tabular}{crrrrrrrrrrrrr}
\hline
\multicolumn{1}{l}{}                                        & \multicolumn{1}{l}{}     & \multicolumn{1}{l}{} & \multicolumn{3}{c}{Short}                                                       & \multicolumn{1}{c}{} & \multicolumn{3}{c}{Medium}                                                      & \multicolumn{1}{c}{} & \multicolumn{3}{c}{Long}                                                        \\ \cline{4-6} \cline{8-10} \cline{12-14} 
$\delta^{\max}$ (min.) & \multicolumn{1}{c}{Obj.} & \multicolumn{1}{c}{} & \multicolumn{1}{c}{Iter.} & \multicolumn{1}{c}{Nodes} & \multicolumn{1}{c}{$1^{\text{st}}$ LB} & \multicolumn{1}{c}{} & \multicolumn{1}{c}{Iter.} & \multicolumn{1}{c}{Nodes} & \multicolumn{1}{c}{$1^{\text{st}}$ LB} & \multicolumn{1}{c}{} & \multicolumn{1}{c}{Iter.} & \multicolumn{1}{c}{Nodes} & \multicolumn{1}{c}{$1^{\text{st}}$ LB} \\ \hline
1                                                           & 131046                   &                      & 41.5                      & 13440                     & 49274                   &                      & 3.7                       & 3989                      & 131039                  &                      & 3.6                       & 3988                      & 131040                  \\
2                                                           & 129980                   &                      & 42.9                      & 13983                     & 48603                   &                      & 5.8                       & 4028                      & 129966                  &                      & 5.1                       & 4022                      & 129967                  \\
3                                                           & 128616                   &                      & 43.2                      & 14360                     & 47961                   &                      & 7.4                       & 4104                      & 128588                  &                      & 6.9                       & 4087                      & 128590                  \\
4                                                           & 127122                   &                      & 48.4                      & 15073                     & 47206                   &                      & 11.8                      & 4298                      & 127068                  &                      & 8.8                       & 4219                      & 127070                  \\
5                                                           & 126226                   &                      & 46.9                      & 15509                     & 45952                   &                      & 12.0                      & 4451                      & 125886                  &                      & 8.4                       & 4261                      & 126173                  \\ \hline
\end{tabular}
}
\end{table}

\subsection{Impact of Refinement Strategy} 

In the next experiment, we solve the 500M and 750M instances with the refinement strategies Duty, Fewer Time Points and Fewer Iterations. We fix the deadheading arc construction scheme to Long, as it proved to be superior in the previous experiment, and will continue using this scheme in all remaining experiments. We again consider a maximum deviation $(\delta^{\max})$ of 1 up to 5 minutes. 

Figure~\ref{fig:compareRS} presents the average computation time for the three refinement strategies. Table~\ref{tbl:compareRS} presents the average number of iterations and number of nodes in the final DDD iteration. We find that the difference in solution time between the strategies is fairly small. Fewer Iterations is the fastest strategy for the majority of the cases, but the difference is limited. We do observe that Fewer Time Points, which minimizes the number of time points that is added to the discretization in an iteration is able to prove optimality on the smallest networks. However, this does not seem to result in a decreased computation time as it comes at the expense of requiring more iterations to prove optimality. In turn, Fewer Iterations results in a larger network, as it adds time points rather aggressively, but typically wins in terms of computation time as it succeeds in requiring fewer iterations.

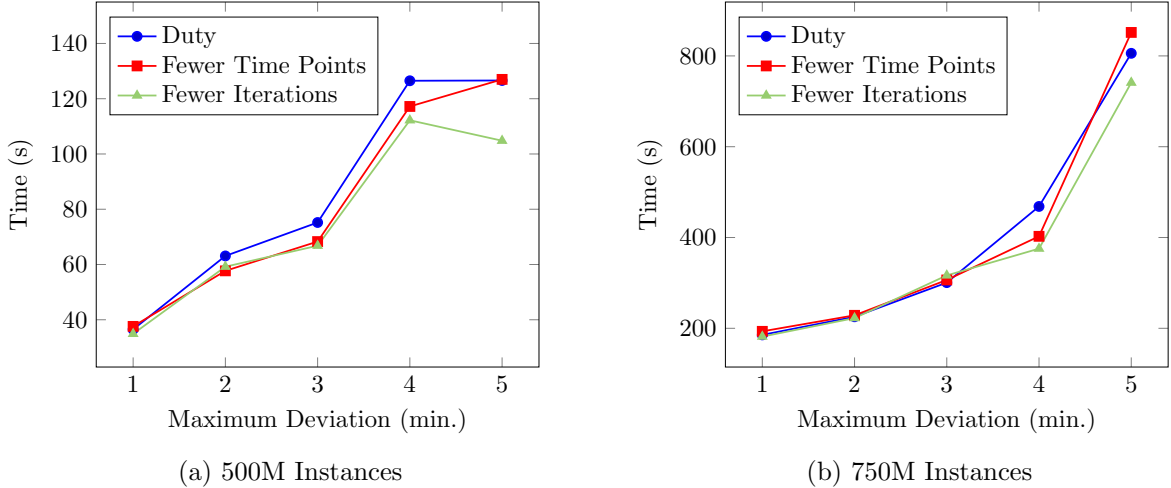
\begin{figure}[ht]
\centering
\begin{subfigure}{0.48\textwidth}
\begin{tikzpicture}[scale=0.85]
    \centering
\begin{axis}[xlabel={Maximum Deviation (min.)},
  ylabel={Time (s)},every axis plot/.append style={thick}, legend style={at={(0.03,0.96)},anchor=north west}, legend cell align={left}, ymax = 155]
\addplot table [x=DeltaMax, y=Path, col sep=comma] {compareRS.csv};
\addlegendentry{Duty}
\addplot+[mark=square*,mark options={fill=red}] table [x=DeltaMax, y=FTP, col sep=comma] {compareRS.csv};
\addlegendentry{Fewer Time Points}
\addplot+[YellowGreen,mark=triangle*,mark options={fill=YellowGreen}] table [x=DeltaMax, y=FI, col sep=comma] {compareRS.csv};
\addlegendentry{Fewer Iterations}
\end{axis}
\end{tikzpicture}
\caption{500M Instances}
\label{fig:compareRSa}
\end{subfigure}
\hfill
\begin{subfigure}{0.48\textwidth}
\begin{tikzpicture}[scale=0.85]
    \centering
\begin{axis}[xlabel={Maximum Deviation (min.)},
  ylabel={Time (s)},every axis plot/.append style={thick}, legend style={at={(0.03,0.96)},anchor=north west}, legend cell align={left}]
\addplot table [x=DeltaMax, y=Path, col sep=comma] {compareRS750.csv};
\addlegendentry{Duty}
\addplot+[mark=square*,mark options={fill=red}] table [x=DeltaMax, y=FTP, col sep=comma] {compareRS750.csv};
\addlegendentry{Fewer Time Points}
\addplot+[YellowGreen,mark=triangle*,mark options={fill=YellowGreen}] table [x=DeltaMax, y=FI, col sep=comma] {compareRS750.csv};
\addlegendentry{Fewer Iterations}
\end{axis}
\end{tikzpicture}
\caption{750M Instances}
\label{fig:compareRSb}
\end{subfigure}
\caption{Average time to optimality for the refinement strategies Duty, Fewer Time Points and Fewer Iterations.}
\label{fig:compareRS}
\end{figure}

\begin{table}[]
\caption{Computational results for the 500M and 750M instances, with different refinement strategies. For each instances, the objective and first lower bound are reported, as well as for each strategy, the average number of iterations and average number of nodes in the final iteration of DDD.} 
\label{tbl:compareRS}
\resizebox{\columnwidth}{!}{%
\begin{tabular}{ccrrrrrrrrrrr}
\hline
\multicolumn{1}{l}{} & \multicolumn{1}{l}{}                                & \multicolumn{1}{l}{}     & \multicolumn{1}{l}{}       & \multicolumn{1}{l}{} & \multicolumn{2}{c}{Duty}                              & \multicolumn{1}{c}{} & \multicolumn{2}{c}{Fewer Time Points}                 & \multicolumn{1}{c}{} & \multicolumn{2}{c}{Fewer Iterations}                  \\ \cline{6-7} \cline{9-10} \cline{12-13} 
Instance             & $\delta^{\max}$ (min.) & \multicolumn{1}{c}{Obj.} & \multicolumn{1}{c}{$1^\text{st}$ LB} & \multicolumn{1}{c}{} & \multicolumn{1}{c}{Iter.} & \multicolumn{1}{c}{Nodes} & \multicolumn{1}{c}{} & \multicolumn{1}{c}{Iter.} & \multicolumn{1}{c}{Nodes} & \multicolumn{1}{c}{} & \multicolumn{1}{c}{Iter.} & \multicolumn{1}{c}{Nodes} \\ \hline
                     & 1                                                   & 131046                   & 131040                     &                      & 3.6                       & 3988                      &                      & 3.4                                     & 3986                                   &                      & 3.4                                    & 3986                                   \\
                     & 2                                                   & 129980                   & 129967                     &                      & 5.1                       & 4022                      &                      & 4.6                                     & 4016                                   &                      & 4.7                                    & 4020                                   \\
500M                 & 3                                                   & 128616                   & 128590                     &                      & 6.9                       & 4087                      &                      & 6.1                                     & 4070                                   &                      & 6.1                                    & 4084                                   \\
                     & 4                                                   & 127122                   & 127070                     &                      & 8.8                       & 4219                      &                      & 8.4                                     & 4180                                   &                      & 8.0                                    & 4222                                   \\
                     & 5                                                   & 126226                   & 126173                     &                      & 8.4                       & 4261                      &                      & 9.3                                     & 4258                                   &                      & 8.3                                    & 4277                                   \\ \vspace{-12pt}\\
 \hline \vspace{-12pt} \\
                     & 1                                                   & 193603                   & 193592                     &                      & 4.6                       & 5981                      &                      & 4.4                                     & 5981                                   &                      & 4.1                                    & 5978                                   \\
                     & 2                                                   & 191118                   & 191007                     &                      & 5.8                       & 6021                      &                      & 5.4                                     & 6004                                   &                      & 5.3                                    & 6016                                   \\
750M                 & 3                                                   & 189619                   & 189597                     &                      & 7.1                       & 6096                      &                      & 6.7                                     & 6074                                   &                      & 6.3                                    & 6109                                   \\
                     & 4                                                   & 187548                   & 187504                     &                      & 10.8                      & 6274                      &                      & 10.3                                    & 6254                                   &                      & 8.4                                    & 6298                                   \\
                     & 5                                                   & 185377                   & 185129                     &                      & 12.0                      & 6545                      &                      & 12.5                                    & 6484                                   &                      & 10.6                                   & 6601                                   \\ \hline
\end{tabular}
}
\end{table}

\subsection{Detailed Performance Large Instances}
In the next experiment, we consider the 1000M instances. We use DDD with Long deadheading arcs, and the Fewer Iterations refinement strategy. As before, we vary $\delta^{\max}$ from 1 to 5 minutes. We use a time limit of 7200 seconds. 

Table~\ref{tbl:largeI} summarizes the performance of DDD on the large instances. We find that within the time limit of 2 hours, the majority of the instances can be solved to optimality. Only for the most complicated instances, with a maximum deviation of 4 and 5 minutes, a small number of instances cannot be solved. However, the average optimality gaps are very small: 4.9 (0.002\%) and 116.2 (0.049\%), respectively. Interestingly, the computation times, iterations and network size do not explode when $\delta^{\max}$ increases. In practice, it is unlikely to have more than 11 potential starting times for every trip. Therefore, this showcases the practicability and scalability of the algorithm. Moreover, we again observe that the initial bound provided by DDD is remarkably close to the final solution. 

To further analyze the behavior of DDD, Figure~\ref{fig:bounds} presents the progress of the average lower and upper bounds over time. We observe that alongside the strong lower bounds, the algorithm is also able to find good upper bounds quickly. Even for the instances with the largest maximum deviation, after ten minutes the gap between lower and upper bound is less than the costs of a single vehicle. 
Relatively speaking, the primal bound is the weaker of the two. While there is hardly any visual improvement in lower bounds over time, the upper bounds do still reduce noticeably over time.  We also observe a tailing-off effect. For example, for $\delta^{\max}$ equal to 2 minutes, the gap is effectively closed after 20 minutes, but it takes over 50 minutes to prove exact optimality.

\begin{table}[]
\caption{Computational results for the 1000M instances for DDD with Long deadheading arcs and Fewer Iterations refinement. For each value of the maximum deviation, the number of optimal solutions, the average objective, absolute gap, computation time, number of iterations, number of nodes and arcs in the final DDD iteration, and the first lower bound are reported.}
\label{tbl:largeI}
\small
\centering
\begin{tabular}{crrrrrrrr}
\hline
$\delta^{\max}$ (min.) & \multicolumn{1}{c}{Opt.} & \multicolumn{1}{c}{Obj.} & \multicolumn{1}{c}{Abs. Gap} & \multicolumn{1}{c}{Time (min.)} & \multicolumn{1}{c}{Iter.} & \multicolumn{1}{c}{Nodes} & \multicolumn{1}{c}{Arcs} & \multicolumn{1}{c}{$1^{\text{st}}$ LB} \\ \hline
1               &  10                  &   249791                     & 0                            & 8.1                             & 5.1                       & 7936                      & 189792                   & 249578                     \\
2               &  10                & 247273                       & 0                            & 16.3                            & 7.3                       & 8073                      & 190827                   & 247246                     \\
3               &  10                  & 243899                       & 0                            & 24.4                            & 10.0                      & 8314                      & 192751                   & 243842                     \\
4               & 9                   & 241203                       & 4.9                          & 35.0                            & 11.3                      & 8749                      & 196525                   & 240929                     \\
5               & 7                   & 238677                        & 116.2                        & 64.7                            & 14.4                      & 9261                      & 201605                   & 238079                     \\ \hline
\end{tabular}
\end{table}

\begin{figure}[ht]
\centering
\begin{subfigure}{0.3\textwidth}
\begin{tikzpicture}[scale=0.55]
    \centering
\begin{axis}[xlabel={Time (min.)},
  ylabel={Costs},every axis plot/.append style={thick}, legend style={at={(0.93,0.93)},anchor=north east}, legend cell align={left},yticklabel style={scaled ticks=false,
                                 /pgf/number format/fixed,
                                 /pgf/number format/precision=0}]
]
\addplot table [x=time, y=ub, col sep=comma] {bounds1.csv};
\addlegendentry{UB}
\addplot+[mark=square*,mark options={fill=red}] table [x=time, y=lb, col sep=comma] {bounds1.csv};
\addlegendentry{LB}
\end{axis}
\end{tikzpicture}
\caption{$\delta^{\max}=1$ min.}
\label{fig:bounds1}
\end{subfigure}
\hspace{10pt}
\begin{subfigure}{0.3\textwidth}
\begin{tikzpicture}[scale=0.55]
    \centering
\begin{axis}[xlabel={Time (min.)},
  ylabel={Costs},every axis plot/.append style={thick}, legend style={at={(0.93,0.93)},anchor=north east}, legend cell align={left},yticklabel style={scaled ticks=false,
                                 /pgf/number format/fixed,
                                 /pgf/number format/precision=0}]
]
\addplot table [x=time, y=ub, col sep=comma] {bounds2.csv};
\addlegendentry{UB}
\addplot+[mark=square*,mark options={fill=red}] table [x=time, y=lb, col sep=comma] {bounds2.csv};
\addlegendentry{LB}
\end{axis}
\end{tikzpicture}
\caption{$\delta^{\max}=2$ min.}
\label{fig:bounds2}
\end{subfigure}
\hspace{10pt}
\begin{subfigure}{0.3\textwidth}
\begin{tikzpicture}[scale=0.55]
    \centering
\begin{axis}[xlabel={Time (min.)},
  ylabel={Costs},every axis plot/.append style={thick}, legend style={at={(0.93,0.93)},anchor=north east}, legend cell align={left},yticklabel style={scaled ticks=false,
                                 /pgf/number format/fixed,
                                 /pgf/number format/precision=0}]
]
\addplot table [x=time, y=ub, col sep=comma] {bounds3.csv};
\addlegendentry{UB}
\addplot+[mark=square*,mark options={fill=red}] table [x=time, y=lb, col sep=comma] {bounds3.csv};
\addlegendentry{LB}
\end{axis}
\end{tikzpicture}
\caption{$\delta^{\max}=3$ min.}
\label{fig:bounds3}
\end{subfigure}
\begin{subfigure}{0.3\textwidth}
\vspace{20pt}
\begin{tikzpicture}[scale=0.55]
    \centering
\begin{axis}[xlabel={Time (min.)},
  ylabel={Costs},every axis plot/.append style={thick}, legend style={at={(0.93,0.93)},anchor=north east}, legend cell align={left},yticklabel style={scaled ticks=false,
                                 /pgf/number format/fixed,
                                 /pgf/number format/precision=0}]
]
\addplot table [x=time, y=ub, col sep=comma] {bounds4.csv};
\addlegendentry{UB}
\addplot+[mark=square*,mark options={fill=red}] table [x=time, y=lb, col sep=comma] {bounds4.csv};
\addlegendentry{LB}
\end{axis}
\end{tikzpicture}
\caption{$\delta^{\max}=4$ min.}
\label{fig:bounds4}
\end{subfigure}
\hspace{20pt}
\begin{subfigure}{0.3\textwidth}
\vspace{20pt}
\begin{tikzpicture}[scale=0.55]
    \centering
\begin{axis}[xlabel={Time (min.)},
  ylabel={Costs},every axis plot/.append style={thick}, legend style={at={(0.93,0.93)},anchor=north east}, legend cell align={left},yticklabel style={scaled ticks=false,
                                 /pgf/number format/fixed,
                                 /pgf/number format/precision=0}]
]
\addplot table [x=time, y=ub, col sep=comma] {bounds5.csv};
\addlegendentry{UB}
\addplot+[mark=square*,mark options={fill=red}] table [x=time, y=lb, col sep=comma] {bounds5.csv};
\addlegendentry{LB}
\end{axis}
\end{tikzpicture}
\caption{$\delta^{\max}=5$ min.}
\label{fig:bounds5}
\end{subfigure}
\caption{Progress of the average lower and upper bound over time for the 1000M instance.}
\label{fig:bounds}
\end{figure}
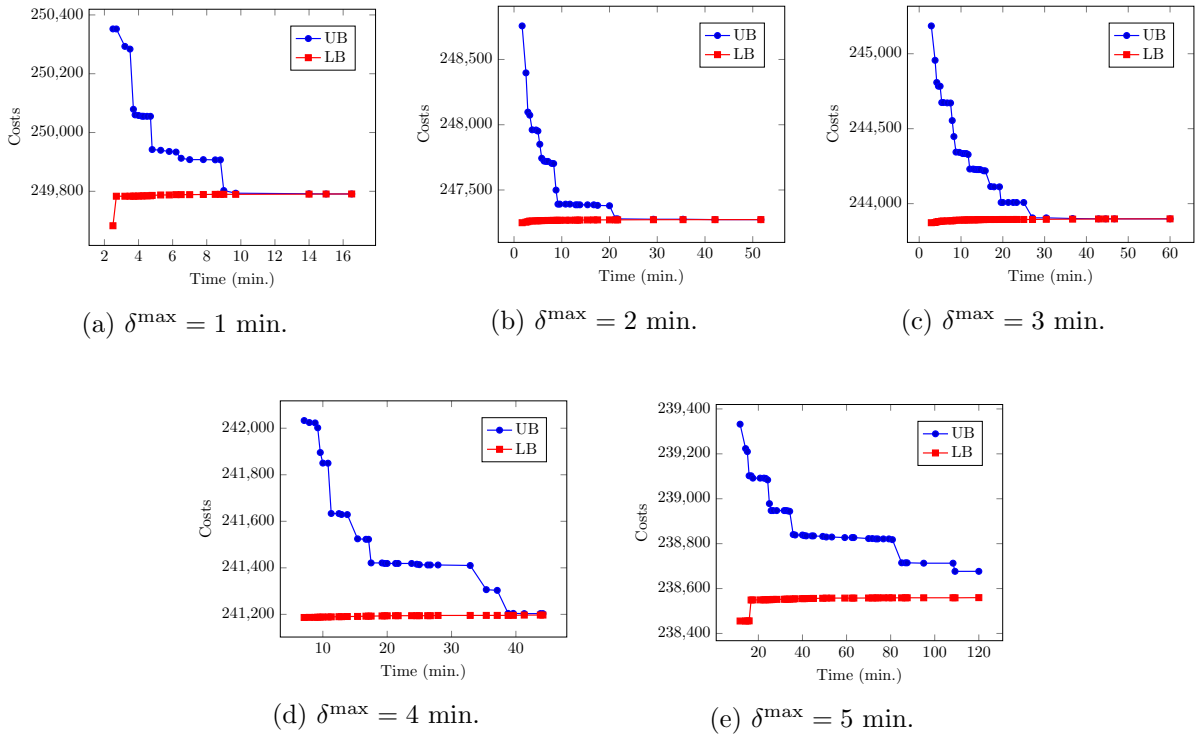

\subsection{Benefits of Trip Shifting} 

In a final experiment, we exploit the strength of the proposed DDD algorithm to provide insights in the benefits of trip shifting, and the trade-off between costs and timetable alterations. We consider all instances of types 500M, 750M and 1000M, and post-process the solutions to minimize unnecessary deviations from the timetable, according to the methods described in Section~\ref{ss:pp}. 

Table~\ref{tbl:gains} presents for all instances the average costs, number of vehicles and deadheading time, and how these change relative to the case where no timetable deviations are allowed. In addition, it reports the average deviation per trip that results from selecting the best duty decomposition, the worst duty decomposition and a randomly selected duty decomposition. Figure~\ref{fig:gains} depicts the average deviation per trip necessary to achieve a certain saving in costs. 

The main observation is that the benefits of trip shifting are considerable. The reduction in costs ranges from about 1\% for a maximum deviation of 1 minute, to about 5\% for a maximum deviation of 5 minutes. These gains also grow in the instance size, likely because having more trips presents more opportunities for combining trips in an efficient way by shifting departure times. In addition,  although the largest share in cost savings come from reducing the number of vehicles, there is also a significant decrease in deadheading time.

\begin{table}[]
\caption{The average costs, number of vehicles, deadheading time, and the average deviation per trip in the best, worst and a randomly selected duty decomposition, for optimal solutions with increasing values of the maximum allowed deviation. }
\label{tbl:gains}
\resizebox{\columnwidth}{!}{%
\begin{tabular}{ccrlrlrlrrr}
\hline
                       & \multicolumn{1}{c}{}          &                          &                                &                              &                                &                             &                                & \multicolumn{3}{c}{Average Dev. (s)}                                           \\ \cline{9-11} 
Instance               & $\delta^{\max}$ (min.) & \multicolumn{1}{c}{Costs} & \multicolumn{1}{c}{(\% diff.)} & \multicolumn{1}{c}{Vehicles} & \multicolumn{1}{c}{(\% diff.)} & \multicolumn{1}{c}{DH-ing} & \multicolumn{1}{c}{(\% diff.)} & \multicolumn{1}{c}{Best} & \multicolumn{1}{c}{Random} & \multicolumn{1}{c}{Worst} \\ \hline
                       & 0                             & 132139                   &                                & 124.1                        &                                & 8039                        &                                & \multicolumn{1}{c}{-}    & \multicolumn{1}{c}{-}      & \multicolumn{1}{c}{-}     \\
                       & 1                             & 131046                   & -0.8                           & 123.1                        & -0.8                           & 7946                        & -1.1                           & 3.5                      & 3.8                        & 3.9                       \\
\multirow{2}{*}{500M}  & 2                             & 129980                   & -1.6                           & 122.1                        & -1.6                           & 7880                        & -2.0                           & 10.8                     & 11.7                       & 12.5                      \\
                       & 3                             & 128616                   & -2.7                           & 120.8                        & -2.7                           & 7816                        & -2.8                           & 22.9                     & 24.0                       & 25.5                      \\
                       & 4                             & 127122                   & -3.8                           & 119.3                        & -3.9                           & 7822                        & -2.7                           & 42.0                     & 44.8                       & 47.2                      \\
                       & 5                             & 126226                   & -4.5                           & 118.5                        & -4.5                           & 7726                        & -3.9                           & 60.0                     & 62.6                       & 65.8                      \\ \hline
                       & 0                             & 195787                   &                                & 184.7                        &                                & 11087                       &                                & -                        & -                          & -                         \\
                       & 1                             & 193603                   & -1.1                           & 182.6                        & -1.1                           & 11003                       & -0.8                           & 3.7                      & 4.0                        & 4.3                       \\
\multirow{2}{*}{750M}  & 2                             & 191118                   & -2.4                           & 180.2                        & -2.4                           & 10918                       & -1.5                           & 11.8                     & 12.6                       & 13.6                      \\
                       & 3                             & 189619                   & -3.2                           & 178.8                        & -3.2                           & 10819                       & -2.4                           & 23.6                     & 25.1                       & 26.7                      \\
                       & 4                             & 187548                   & -4.2                           & 176.8                        & -4.3                           & 10748                       & -3.1                           & 41.4                     & 43.7                       & 46.6                      \\
                       & 5                             & 185377                   & -5.3                           & 174.7                        & -5.4                           & 10677                       & -3.7                           & 62.2                     & 65.6                       & 69.6                      \\ \hline
                       & 0                             & 252975                   &                                & 238.8                        &                                & 14175                       &                                & -                        & -                          & -                         \\
                       & 1                             & 249791                   & -1.3                           & 235.7                        & -1.3                           & 14091                       & -0.6                           & 4.4                      & 4.7                        & 5.2                       \\
\multirow{2}{*}{1000M} & 2                             & 247273                   & -2.3                           & 233.3                        & -2.3                           & 13973                       & -1.4                           & 13.9                     & 15.1                       & 16.1                      \\
                       & 3                             & 243899                   & -3.6                           & 230.0                        & -3.7                           & 13899                       & -1.9                           & 28.4                     & 30.1                       & 32.2                      \\
                       & 4                             & 241203                   & -4.7                           & 227.4                        & -4.8                           & 13803                       & -2.6                           & 45.4                     & 47.9                       & 51.2                      \\
                       & 5                             & 238688                   & -5.6                           & 225.0                        & -5.8                           & 13688                       & -3.4                           & 66.2                     & 69.6                       & 72.8                      \\ \hline
\end{tabular}
}
\end{table}

\begin{figure}[ht]
\centering
\begin{tikzpicture}[scale=0.95]
    \centering
\begin{axis}[xlabel={Average Deviation (s)},
  ylabel={Cost Reduction (\%)},every axis plot/.append style={thick}, legend style={at={(0.05,0.95)},anchor=north west},legend cell align={left}]
\addplot table [x=Dev1000, y=Gains1000, col sep=comma] {gains.csv};
\addlegendentry{1000M}
\addplot+[mark=square*,mark options={fill=red}] table [x=Dev750, y=Gains750, col sep=comma] {gains.csv};
\addlegendentry{750M}
\addplot+[YellowGreen,mark=triangle*,mark options={fill=YellowGreen}] table [x=Dev500, y=Gains500, col sep=comma] {gains.csv};
\addlegendentry{500M}
\end{axis}
\end{tikzpicture}
\caption{Trade-off between reduction in costs and the average deviation per trip with using best duty decomposition.}
\label{fig:gains}
\end{figure}
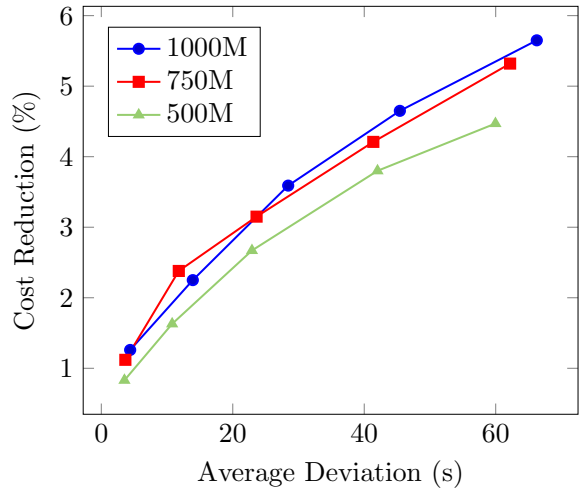

Another finding from both the table and figure is that the average timetable alterations per trip necessary to achieve these reductions in costs are relatively small. For example, for $\delta^{\max}$ equal to 1 minute, the average deviation ranges from about 3 to 5 seconds, which implies that up to one of every twelve trips is shifted by 1 minute. The average deviation does increase in $\delta^{\max}$ and in the instance size, but there we of course also observe larger cost savings. Even if the maximum deviation is 5 minutes, the average deviation is closer to 1 minute. Finally, we find that optimizing the duty decomposition as a post-processing step has clear benefits: it typically drives down the required timetable alterations by over 5\% compared to choosing a random decomposition, and by over 10\% compared to the worst decomposition. 

\section{Conclusion}
\label{s:conclusion}
In this paper, we have presented a tailored Dynamic Discretization Discovery algorithm for the Multi-Depot Vehicle Scheduling Problem with Trip Shifting. Inspired by \cite{boland2017continuous}, this iterative refinement method is able to find an optimal solution to the MDVSP-TS without explicitly modeling all possible trip shifts. Instead, it iteratively solves an integer program on a partially time-expanded network and extends the network only where it causes infeasibility. The core idea behind DDD is that solving the problem on a relaxed network multiple times yields an optimal solution much quicker than through using the fully time-expanded network only once.

Our computational study has shown that DDD with either Medium or Long deadheading arcs indeed outperforms the explicit modeling approach by a wide margin. While the size of the fully time-expanded network quickly grows to a point where the model is intractable when the maximum allowed deviation increases, DDD is able to solve large instances even for large deviations. The performance of the algorithm is relatively insensitive to the choice of refinement strategy, but in most cases Fewer Iterations, targeted at limiting the number of DDD iterations by adding time points aggressively, was able to prove optimality the quickest. Besides, the algorithm provides strong upper bounds already in the early stages of the algorithm, whereas the lower bound is close to the optimal objective value from the first iteration onward. Combined with the observation that the problem becomes increasingly difficult to solve when the iterations progress, it may suffice to terminate the algorithm after a few iterations to obtain a good solution relatively quickly.



All in all, DDD has proven to be a strong addition to the toolbox for solving the MDVSP-TS. Interesting ideas for future research include extending the application of DDD to electric vehicle or integrated vehicle and crew scheduling problems.

\bibliographystyle{apacite}
\bibliography{mybib}

\end{document}